\documentclass{amsart}
\usepackage{latexsym}
\usepackage{amsmath}
\usepackage{amssymb}
\usepackage[all]{xy}

\newtheorem{thm}{Theorem}[section]
\newtheorem{prop}[thm]{Proposition}

\newtheorem{lem}[thm]{Lemma}

\newtheorem{assertion}[thm]{Assertion}
\theoremstyle{definition}

\theoremstyle{remark}
\newtheorem{rem}[thm]{Remark}

\newcommand{\K}{{\mathbb K}}
\newcommand{\Q}{{\mathbb Q}}
\newcommand{\calS}{{\mathcal S}}
\newcommand{\B}{{\mathcal B}}
\newcommand{\M}{{\mathcal M}}
\newcommand{\e}{\varepsilon}
\newcommand{\m}{\mathsf{m}}
\newcommand{\mapright}[1]{%
 \smash{\mathop{%
  \hbox to 1cm{\rightarrowfill}}\limits_{#1}}}
\newcommand{\maprightd}[2]{%
 \smash{\mathop{%
  \hbox to 1.2cm{\rightarrowfill}}\limits^{#1}\limits_{#2}}}
\newcommand{\mapleft}[1]{%
 \smash{\mathop{%
  \hbox to 1cm{\leftarrowfill}}\limits_{#1}}}
\newcommand{\mapleftu}[1]{%
 \smash{\mathop{%
  \hbox to 0.8cm{\leftarrowfill}}\limits^{#1}}}
\newcommand{\maprightu}[1]{%
 \smash{\mathop{%
  \hbox to 1cm{\rightarrowfill}}\limits^{#1}}}
\newcommand{\maprightud}[2]{%
 \smash{\mathop{%
  \hbox to 1cm{\rightarrowfill}}\limits^{#1}_{#2}}}
\newcommand{\mapleftud}[2]{%
 \smash{\mathop{%
  \hbox to 1cm{\leftarrowfill}}\limits^{#1}_{#2}}}

\newcounter{eqn}[section]

\def\theeqn{\textnormal{(\thesection.\arabic{eqn})}}

\def\eqnlabel#1{%
  \refstepcounter{eqn}%
  \label{#1}%
  \leqno{\theeqn}}

\begin{document}

\title[On the whistle cobordism operation]{On the whistle cobordism operation in string topology of classifying spaces 
}

\footnote[0]{{\it 2010 Mathematics Subject Classification}: 
55P50, 81T40, 55R35
\\ 
{\it Key words and phrases.} String topology, classifying space, topological quantum field theory, Eilenberg-Moore spectral sequence.  

Department of Mathematical Sciences, 
Faculty of Science,  
Shinshu University,   
Matsumoto, Nagano 390-8621, Japan   
e-mail:{\tt kuri@math.shinshu-u.ac.jp}
}

\author{Katsuhiko KURIBAYASHI}
\date{}
   
\maketitle

\begin{abstract} In this manuscript, we consider cobordism operations in the $2$-dimensional labeled open-closed topological quantum field theory for the classifying space of 
a connected compact Lie group in the sense of Guldberg. In particular, it is proved that the whistle cobordism operation is non-trivial in general provided the labels are in the set of maximal closed subgroups of the given Lie group.  
The non-triviality of cobordism operations induced by gluing the whistle with the opposite and other labeled cobordisms is also discussed. 
\end{abstract}

\section{Introduction} 
String topology initiated by Chas and Sullivan \cite{C-S} 
until now provides many impressive and fruitful algebraic structures for the homology of the free loop spaces of orientable manifolds 
\cite{C-J, C-G, Godin}, orbifolds \cite{L-U-X}, the classifying spaces of Lie groups \cite{C-M, H-L, K-M},  
Gorenstein spaces \cite{F-T, K-L-N} and differentiable stacks \cite{B-G-N-X}. 
In this manuscript, we deal with string topology for classifying spaces,   
which is enriched with a $2$-dimensional  labeled open-closed {\it topological quantum field theory} (TQFT). 

In \cite{G},  Guldberg has developed such a labeled TQFT for the classifying space $BG$ of a connected compact Lie group $G$. 
In consequence, the homology groups of double coset spaces associated to $G$ and of the free loop space $LBG$ 
are simultaneously incorporated into the open-closed TQFT with 
labels in a set of closed connected subgroups of $G$. The structure is indeed induced by an open-closed {\it homological conformal field theory} (HCFT) which is 
an extended version of a closed HCFT for classifying spaces due to Chataur and Menichi \cite{C-M}; see \cite[Theorem 1.2.3 and Lemma 2.4.1]{G}. 
The aim of this manuscript is to investigate the non-triviality of an important cobordism operation which connects the open and closed theories in the labeld TQFT. 


To describe a labeled open-closed TQFT in general, we need to introduce the category 
$\mathsf{oc\text{-}Cobor}(\calS)$ of open-closed strings. Its  objects are finite disjoint unions of oriented circle and intervals with ends labeled by elements of a fixed set $\calS$. A morphism in the category 
is the diffeomorphism class of cobordisms from $Y_0$ to $Y_1$ labeled in $\calS$, where such a cobordism is indeed a 
$2$-dimensional oriented manifold $\Sigma$ whose boundary consists of three parts 
$
Y_0\cup Y_1\cup \partial_{\text{free}}\Sigma.
$
Here the part $\partial_{\text{free}}\Sigma$ called the {\it free boundary} is a cobordism 
between $\partial Y_0$ and $\partial Y_1$. Moreover, it is required that the connected component of  
$\partial_{\text{free}}\Sigma$ is labeled by elements of $\calS$ compatible with the labeling of 
$\partial Y_0$ and $\partial Y_1$; see \cite[Section 2]{M-S} for more details. 
The compositions are given by gluing cobordisms provided the labelings of the boundaries are compatible.  
By definition,  a {\it $2$-dimensional labeled open-closed  topological quantum field theory} is a monoidal functor $\mu$ from $(\mathsf{oc\text{-}Cobor}(\calS), \coprod)$ to $(\mathbb{K}\text{-}\mathsf{Vect}, \otimes)$ 
the category of graded vector space over $\K$, where the monoidal structure  $\coprod$ of $\mathsf{oc\text{-}Cobor}(\calS)$ is given by the disjoint union of cobordisms. 
We may write $\mu_\Sigma$ for the linear map assigned by a cobordism $\Sigma$.  Moreover, we denote by $(\Sigma, \{\Sigma^H\}_{H\in \calS'})$ 
a one or two dimensional labeled cobordism whose free boundary has conected components $\{\Sigma^H\}_{H\in \calS'}$, where $\calS'$ is a subset of 
$\calS$.

While the labeled open-closed TQFT for a classifying space is investigated in this article, we refer the reader to the results due to Blumberg, Cohen and Teleman \cite{BCT}
for an open-closed TQFT labeled by submanifolds of a given manifold; see also \cite{Godin} for an HCFT based on the free loop space of a manifold. 

Let $G$ be a compact connected Lie group and $\B$ a set consisting of connected closed subgroups of $G$. 
Let $W=(W, \{W^H\})$ denote the whistle cobrodism from the interval $I$ to the circle $S^1$ whose 
incoming boundary $I$ is connected with an arc $W^H$ labeled by a subgroup $H \in \B$ at the each endpoint; 
see the figure below for the whistle cobordism.  
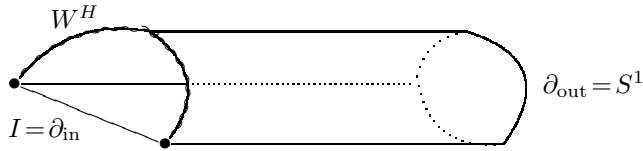
\[
\hspace{1.5cm}
\begin{xy} 
(4, -6) *{I \!= \!\partial_{\text{in}}}, 
(77, 0) *{\partial_{\text{out}}\!=\! S^1}, 
(8, 9) *{W^H},
(0,0) *{\bullet}="1",
(20,-8)*{\bullet}="2", 
(23, 0) *{}="7",
(18, 7) *{}="3", 
(52.5, 0)*{}="4",
(53.5, 1)*{}="4.5",
(65,-8)*{}="5", 
(60, 7) *{}="6", 
\ar @{-}"1";"2", 
\ar @{-}"1";"7", 
\ar @{..}"7";"4", 
\ar @{-}"2";"5", 
\ar @{-}"3";"6",
\ar @/^3mm/ @{~}"1";"3",
\ar @/^4mm/ @{~}"3";"2",
\ar @/^3mm/ @{-}"1";"3",
\ar @/^4mm/ @{-}"3";"2",
\ar @/^5.5mm/ @{-}"6";"5",
\ar @/_3.5mm/ @{..}"4.5";"5",
\ar @/_1.5mm/ @{..}"6";"4.5",
\end{xy}
\]

\noindent
Observe that the arc $W^H$, which is denoted by the wave curve, is the only free boundary of the whistle.  In \cite{G}, the non-triviality of a cobordism operation of open strings, namely intervals, is revealed; see Appendix A for more computations in an open TQFT. 
Since the whistle cobrodism connects open and closed strings, it is anticipated that the operation associated with the whistle plays a key role in the open-closed TQFT.  
In fact, an open-closed TQFT splits into the open theory and the closed one if all whistle cobordism operations are trivial; see \cite[Propositions 3.8 and 3.9]{L-P} for generators of morphisms in $\mathsf{oc\text{-}Cobor}(\B)$ for example. 
In this article, we focus on the whistle cobordism operation and the non-triviality is discussed. 

In what follows, the homology and cohomology are with 
coefficients in a field $\K$. Our main theorem is described as follows. 

\begin{thm}\label{thm:main}
Let $G$ be a connected compact Lie group and $H$ a connected closed subgroup of maximal rank. 
Suppose that the integral homology groups of $G$ and $H$ are $p$-torsion free, where $p$ is the characteristic of $\K$.  Then the operations 
$\mu_W$ and $\mu_{W^{\text{\em op}}}$ associated to the whistle cobordisms $(W, \{W^H\})$ and $(W^{\text{\em op}}, \{{(W^{\text{\em op}}})^H\})$ 
are non-trivial. 
Moreover, the composite operation $\mu_{W} \circ \mu_{W^{\text{\em op}}}=\mu_{W\circ W^{\text{\em op}}}$ 
is also non-trivial if 
$(\deg (B\iota)^*(x_i), p)=1$ for any $i = 1, ..., l$, where $B\iota : BH \to BG$ stands for the map between classifying spaces induced by the inclusion 
$\iota : H \to G$ and 
$x_1, ..., x_l$ are generators of $H^*(BG; \K)$.  
\end{thm}

This enables us to conclude that the open theory and closed theory are inseparable in general. In consequence, we 
can explicitly determine every labeled open-closed cobordism operation over the rational $\mathbb{Q}$ 
under an appropriate assumption on the set of  labels; see Assertion \ref{assertion:TQFT}.  

We observe that the cobordism $W\circ W^{\text{op}}$ in Theorem \ref{thm:main} is the cylinder with a hole labeled by $H$ 
a subgroup of maximal rank. Operations associated with other composites of the whistle cobordisms are discussed in Remark \ref{rem:OtherOperations} below.  

The rest of the article is organized as follows. In Section 2, we briefly review the construction of 
a cobordism operation in the labeled TQFT for classifying spaces due to Guldberg. After describing our strategy for proving the main theorem, we consider the Eilebnerg-Moore spectral sequences for appropriate pullback diagrams and give the proof in Section 3.  Appendix A deals with a labeled open TQFT for classifying spaces. 
In Appendix B, we discuss rational homotopy theoretical methods for computing the whistle cobordism operation.

\section{A brief review of the labeled TQFT for classifying spaces}
We recall the cobordism operation introduced in \cite[Section 2.3]{G}. In what follows, we denote by $\text{map}(X, Y)$ 
the mapping space of maps from $X$ to $Y$ with compact-open topology. Observe that $\text{map}(S^1, BG) = LBG$ by definition. 
For a two dimensional labeled cobordism $\Sigma :=(\Sigma, \{\Sigma^H\}_{H\in \B})$ with in-coming boundary $\partial_{\text{in}}$ and outgoing boundary $\partial_{\text{out}}$, we define a space $\M(\Sigma)$ by the pullback diagram
\[
\xymatrix@C35pt@R18pt{
\M(\Sigma) \ar[r] \ar[d] \ar[r] & \text{map}(\Sigma, BG) \ar[d]^{i^*}\\
\prod_H \text{map}(\Sigma^H, BH) \ar[r]_{B\iota_*} & \prod_H \text{map}(\Sigma^H, BG),
}
\] 
where $\iota : H \to G$ is the inclusion and $i : \coprod_H \Sigma^H=\partial_{\text{free}}\Sigma \to \Sigma$ denotes the embedding. By applying the same pullback construction as above to a one dimensional cobordism of the form 
$ \partial_{\text{\em in}} = (\partial_{\text{\em in}},  \{\Sigma^H\cap \partial_{\text{\em in}}\}_{H\in \B})$, 
we define a space $\M(\partial_{\text{\em in}})$. The naturality of the construction enables us to obtain a map 
$in^* : \M(\Sigma) \to \M(\partial_{\text{\em in}})$ from the inclusion $in : \partial_{\text{\em in}} \to \Sigma$ of the 
in-coming boundary.

\begin{prop}\cite[Proposition 2.3.9]{G} \label{prop:fibrations} {\em (i)} The map $in^*$ induced the inclusion gives rise to  a fibration $\M(\Sigma)_c\to \M(\Sigma) \stackrel{in^*}{\to} \M(\partial_{\text{\em in}})$ 
whose fibre $\M(\Sigma)_c$ is the product of $\Omega BH \simeq H$, $G/H$ and 
the total space $E$ of a fibration of the form $\Omega BH \to E \to G/H$ in which $H$'s are the labels of the cobordism $\Sigma$. \\
{\em (ii)} The fibration in {\em (i)} is orientable; that is, the action of the fundamental group of the base on the homology of the fibre is trivial. 
\end{prop}

\noindent
Thus for the fibration $h :={in^*} : \M(\Sigma) \to \M(\partial\Sigma)$, we can define the integration along the fibre $h_! : H_*(\M(\partial\Sigma)) \to H_{*+i}(\Sigma)$ 
with degree $i$ the top degree of $H_*(\M(\Sigma)_c)$.  The main result \cite[Theorem 1.2.3]{G} asserts that 
an operation $\mu_\Sigma$ defined by the composite 
\[
\mu_\Sigma :  H_*(\M(\partial_{\text{in}})) \stackrel{h_!}{\longrightarrow}  H_*(\M(\Sigma)) \stackrel{(out^*)_*}{\longrightarrow} 
 H_*(\M(\partial_{\text{out}}))
\]
for each labeled cobordism $\Sigma$ gives rise to a labeled open-closed TQFT structure for the classifying space of $G$.  
We here and henceforth omit the action of {\it determinants} in the TQFT for classifying spaces; see \cite{C-M, G} for the action. This means that our computation of the cobordism operations below is made up to multiplication by non-zero scalar.

For a labeled whistle cobordism $W =(W, W^H)$, let $a$ and $b$ be the two endpoints of the arc $W^H$ and hence they are also endpoints of the in-boundary 
$\partial_{\text{in}}$ of $W$. 
In what follows, we may write ${}_a\cap_b$ and ${}_a\text{---}_b$ for the arc $W^H$ and the in-boundary $\partial_{\text{in}}$, respectively.

\begin{rem}\label{rem:fibre}
As for the whistle $W =(W, W^H)$, the fibration $h=in^* : \M(W) \to \M(\partial_{\text{in}})$ induced by the embedding $\partial_{\text{in}} \to W$ 
is the homotopy pullback of 
$$in^* : \M(W^H)=\text{map}(W^H, BG) \to 
\M(\{a, b\})=\text{map}(\{a, b\} ,BH)$$ 
along the map $(in')^* : \M(\partial_{\text{in}}) \to \M(\{a, b\})$. 
Thus the map $in_*$ is regarded as the evaluation map $BH^I \to BH\times BH$ at $0$ and $1$. We see that the fibre of $h$ has the homotopy type of $\Omega BH$ and 
then of $H$. Moreover, the fibration $k={out}_* : \M(W) \to \M(\partial_{\text{out}})$ is the homotopy pullback of the fibration $B\iota : BH \to BG$ along the evaluation map 
$ev_0 : \M(\partial_{\text{out}})=LBG \to BG$. This implies that the fibre of $k$ has the homotopy type of the homogeneous space $G/H$.  
These results follow from the proof of \cite[Proposition 2.3.9]{G}. We observe that $\deg \mu_W = \text{dim} \ H$ and 
$\deg \mu_{W^{\text{op}}}=  (\text{dim} \ G/H)=  (\text{dim} \ G -\text{dim} \ H)$. 

If labels are in the set of subgroups of maximal rank, then Grassmann manifolds and flag manifolds may appear as  
the fibres of the fibrations $k : \M(W) \to \M(\partial_{\text{out}})$. 
\end{rem}

\section{Proof of the main result}

We begin by describing our strategy for proving Theorem \ref{thm:main}. 

\begin{enumerate}
\setlength{\parskip}{0cm} 
  \setlength{\itemsep}{-0.4cm}
\item We deal with the dual operation $D\mu_W= h^! \circ k^*$ on the cohomology. \\
\item Determine explicitly the cohomology algebras of $\M(\partial W)$ and $\M(W)$ with the Eilenberg-Moore spectral sequences, and investigate the behavior of the maps 
$h^*$ and $k^*$ for generators of the cohomology comparing the spectral sequences. \\
\item Consider the Leary-Serre spectral sequence for the fibration in Proposition \ref{prop:fibrations} in order to compute the integration $h^!$ along the fibre. \\
\item  Determine the value of the composite $h^! \circ k^*$ at an appropriate element of $H^*(\M(\partial_{\text{out}}))$. \\
\item As for the latter half of the assertions, we also consider the dual operation $D\mu_{W^{\text{op}}}= k^! \circ h^*$ with the same strategy as above. \\
\item Reveal the nontriviality of the composite $D\mu_{W^{\text{op}}}\circ D\mu_W = D(\mu_{W\circ W^{\text{op}}})$ with the description of the fundamental class of the homogeneous space $G/H$ due to Smith \cite{S}. 
\end{enumerate}

We recall the Eilenberg-Moore spectral sequence (EMSS) in a general setting. 
Let $p : E \to B$ be a fibration over a simply-connected base $B$ and $\pi : E_\varphi \to X$ 
the pullback along a map $\varphi : X \to B$. We then have the Eilenberg-Moore spectral sequence \cite{S_2} converging 
to the cohomology $H^*(E_\varphi)$ as an algebra with 
$
E_2^{*,*} \cong \text{Tot}^{*,*}_{H^*(B)}(H^*(E), H^*(X))
$
as a bigraded algebra. Observe that each term of the spectral sequence appears in the second quadrant. 

In order to compute the cohomology concerning the whistle cobordism operation by using the EMSS,  
we consider commutative diagrams 
$$
\xymatrix@C10pt@R8pt{
  & \M(\partial_{\text{out}}) \ar[rr]^{i}  \ar@{->}'[d]^{}[dd]  
  & & \text{map}(S^1, BG) \ar[dd]^{} & \\
\M(W) \ar[rr]_(0.6){j} \ar[ru]^(0.4){k:=(out)^*} \ar[dd] & & \text{map}(W, BG) \ar[ru]_{\simeq}^{res} \ar[dd]^(0.3){res} & \\
   & \text{map}(\varnothing, BG) \ar@{=}'[r][rr] & &   \text{map}(\varnothing, BG) & \\
\text{map}({}_a\cap_b, BH) \ar[rr]_{(B\iota)_*} \ar[ur]^{} & & \text{map}({}_a\cap_b, BG),  \ar[ru] &
}
\eqnlabel{add-0}
$$
$$
\xymatrix@C10pt@R8pt{
 &  \M(W) \ar[ld]_-{h} \ar@{->}'[d][dd]  \ar[rr]^{j} & &  \text{map}(W, BG) \ar[ld]^{res} \ar[dd]^(0.6){res} \\
 \M(\partial_{\text{in}})  \ar[rr]  \ar[dd]_\beta & & \text{map}({}_a\text{---}_b, BG) \ar[dd]^(0.35){p}  \\
   & \text{map}({}_a\cap_b, BH)\ar@{->}'[r]^(0.7){(B\iota)_*}[rr] \ar[ld]_{res}& & \text{map}({}_a\cap_b, BG) \ar[ld]_{res} \\
\text{map}(\{a, b\}, BH)\ \ar[rr]_{(B\iota)_*}  & & \text{map}(\{a, b\}, BG) 
}
\eqnlabel{add-1}
$$
in which the front and back squares are pullback diagrams, and $res$ and $p$ denote the maps induced by the embeddings. 
Moreover, using a deformation retraction $r : W \to  \partial_{\text{out}}$ with $r(a)=0=r(b)$ and 
$r({}_a\text{---}_b) =  \partial_{\text{out}}$, which is a homotopy inverse of the embedding $out : \partial_{out} \to W$, we have commutative diagrams 
$$
\xymatrix@C5pt@R8pt{
                                    &  \text{map}(W, BG) \ar[ld]_{res} \ar[dd]^(0.35){res} &&  \text{map}(S^1, BG) \ar[dd]^{ev_0}\ar[ll]_{r^*}^{\simeq} \ar[rr]^-u && BG^I \ar[dd]^-{\e_0\times \e_1 }\\
 \text{map}({}_a\text{---}_b, BG) \ar[dd]_(0.5){res} \ar@/_0.8pc/[rrrrru]_(0.5){=} & \\
    & \text{map}({}_a\cap_b, BG) \ar[ld]_{res} 
                         && BG  \ar[ll]^-{\simeq }_-{t^*}\ar[rr]^-{\Delta}&& BG\times BG, \\
  \text{map}(\{a, b\}, BG) \ar@/_0.8pc/[rrrrru]_{=} & 
}
\eqnlabel{add-2}
$$
in which the back right square is a pullback diagram, where $t : {}_a\cap_b \to \{a\}$ denotes a deformation retraction. 

For morphisms $f : A \to M$ and $g : A\to N$ of algebras, we can regard $M$ and $N$ as $A$-modules via the morphisms. Then we write $\text{Tor}_A(M,N)_{f,g}$ for the torsion product of $M$ and $N$.  By the assumption of the theorem, 
we see that the cohomology algebras of $BG$ and $BH$ are polynomial, say $H^*(BG)\cong \K[x_i, ..., x_l]$ and $H^*(BG)\cong \K[u_i, ..., u_l]$; see \cite{Mimura-Toda}. Applying the EMSS to the pullback diagrams in (3.1) and (3.2), we have commutative diagrams  
$$
{\small 
\xymatrix@C15pt@R20pt{
H^*(\M(\partial_{\text{out}})) \ar[d]_{k^*} &\text{Tor}_{\K}(\K, H^*(BG^{S^1}))=:A   \ar[d]_{\text{Tor}_\eta(\eta, (out^*)^*)}  \ar[l]_-{\cong}^-{EM}\\
H^*(\M(W)) &\text{Tor}_{H^*(BG^\cap)}(H^*(BH^\cap), H^*(BG^W)))_{(B\iota_*)^*, res^*}=:B \ar[l]_-{\cong}^-{EM} \\
H^*(\M(\partial_{\text{in}})) \ar[u]^{h^*}& \text{Tor}_{H^*(BG^{\{a, b\}})}(H^*(BH^{\{a, b\}}), H^*(BG^{{}_a\text{---}_b})))_{(B\iota_*)^*, p^*}=:C, \ar[l]_-{\cong}^-{EM}  
\ar[u]^{\text{Tor}_{res^*}(res^*, res^*)}
}
}
\eqnlabel{add-3}
$$
where $\eta$ is the unit, $X^K$ is the mapping space $\text{map}(K, X)$ and $\cap$ denotes the arc ${}_a\cap_b=W^H$. In fact, the argument with the {\it two sided Koszul resolution} implies 
that each spectral sequence collapses at the $E_2$-term. We observe that the resolution is of the form 
\[
\xymatrix@C20pt@R17pt{
(H^*(BG)\otimes H^*(BG)\otimes \wedge(y_1, ..., y_l), D) \ar[r]^-{\m} &H^*(BG)\cong H^*(BG^I) \ar[r] & 0,
}
\]
where $D(y_i)=x_i\otimes 1 -1\otimes x_i$, $\text{bideg} \  y_i = (-1, \deg x_i)$ and $\m$ stands for the multiplication on $H^*(BG)$; see \cite{B-S}.
There exists no extension problem in the spectral sequences; see the diagram (3.6) below. Thus the naturality of each isomorphism $EM : \text{Tot}E_2^{*,*} \to   \text{Tot}E_\infty^{*,*} \cong (\text{the target cohomology})$, which is induced by the Eilenberg-Moore map \cite{S_2}, allows us to obtain the commutative diagram. Moreover, by using the retraction $r$ mentioned above, we have commutative diagrams
$$
{\small 
\xymatrix@C30pt@R17pt{
A \ar[d]_{}  \ar[r]_{=} & H^*(BG^{S^1})=:A'  \ar[d]_{\iota\otimes 1} \\
B  \ar[r]_-{\cong}^-{\text{Tor}(1, (r^*)^*)} &\text{Tor}_{H^*(BG^\cap)}(H^*(BH^\cap), H^*(BG^{S^1})))_{(B\iota_*)^*, (r^*)^*\circ res^*} =:B'  \\
C \ar[u]^{} \ar[r]_-{\cong}  &  \text{Tor}_{H^*(BG^{\times 2})}(H^*(BH^{\times 2}), H^*(BG^{{}_a\text{---}_b}))_{((B\iota_*)^*)^{\otimes 2}, res^*}=:C', 
 \ar[u]^{\text{Tor}(res^*, (r^*)^*\circ res^*)} 
}}
\eqnlabel{add-4}
$$
where the left vertical arrows are the same ones as in (3.4). 
Since the back left diagram in (3.3) is commutative, it follows that $(r^*)^*\circ res^*=(ev_0)^*\circ (t^*)^*$. Moreover, the diagram (3.3) enables us to deduce the commutativity of 
the left-hand side two squares in the diagram 
$$
{\small 
\xymatrix@C18pt@R20pt{
A'  \ar[d]_{}  & \text{Tor}_{H^*(BG^{\times 2})}(H^*(BG), H^*(BG^I)))_{\Delta^*, (\e_0\times e_1)^*} \ar[l]_-{\cong}^-{EM} \ar[d] 
            & H^*(BG)\otimes \wedge (y_1, ..., y_l) \ar[l]_-{\cong} \ar[d]^{(B\iota)^*\otimes 1}\\
B' \ar[r]^-{\cong}_-{\text{Tor}_{t^*}(t^* \!, 1)} &\text{Tor}_{H^*(BG)}(H^*(BH), H^*(BG^{S^1})))_{(B\iota_*)^*, (ev_0)^*} &  H^*(BH)\otimes \wedge (y_1, ..., y_l) \ar[l]_-{\cong}\\
C' \ar[u]^{} \ar[r]_-{\cong}  &  \text{Tor}_{H^*(BG^{\times 2})}(H^*(BH^{\times 2}), H^*(BG^I)_{((B\iota_*)^*)^{\otimes 2}, res^*} \ar[u]^-{\text{Tor}_{\Delta^*}(\Delta^*, u^*)} 
           &  \frac{H^*(BH)\otimes H^*(BH)}{((B\iota)^*x_i\otimes 1 - 1\otimes (B\iota)^*x_i )}. \ar[l]_-{\cong} \ar[u]_{\m}
}}
\eqnlabel{add-5}
$$
Explicit calculations of the EMSS's with the Koszul resolutions above give the commutative diagrams in the right-hand side in (3.6), 
where $\m$ denotes the map induced by the multiplication of the algebra $H^*(BH)$. 
We observe that $$(B\iota)^*x_1\otimes 1 - 1\otimes (B\iota)^*x_1, ..., (B\iota)^*x_l\otimes 1 - 1\otimes (B\iota)^*x_l$$ give a regular sequence since $H$ is of maximal rank. 

We are ready to prove main theorem. 

\begin{proof}[Proof of the non-triviality of $\mu_W$] In order to compute the integration along the fibre associated with 
the fibration $h := in^* : \M(W) \to \M(\partial_{\text{in}})$, we consider the Leray-Serre spectral sequence $\{{_{LS}}E_r^{*,*}, d_r \}$ for the fibration. 
As mentioned in Remark \ref{rem:fibre}, the fibration $h$ fits into the commutative diagram 
\[
\xymatrix@C15pt@R20pt{
\M(W) \ar[r] \ar[d] _h & BH^I \ar[d]^{\e_0\times \e_1} \ar[r] & K({\mathbb Z}/p, \deg u_i)^I \ar[d]^{\e_0\times \e_1}\\
\M(\partial_{\text{in}}) \ar[r]_-\beta & BH \times BH \ar[r]_-{f_i}& K({\mathbb Z}/p, \deg u_i) \times K({\mathbb Z}/p, \deg u_i)
}
\]
in which the left-hand side square is a pullback, where $\beta$ is the map in (3.2) and the map $f_i$ represents the element $u_i$.  
The argument with the EMSS yields that $H^*(\Omega BH) \cong H^*(H) \cong \wedge (z_1, ..., z_l)$ as algebras, where $\deg z_i =\deg u_i-1$. 
Comparing the Leray-Serre spectral sequences of the right two fibrations, we see that $z_i$ is transgressive to $u_i\otimes 1 - 1\otimes u_i$ for $i= 1, ..., l$ 
in the Leray-Serre spectral sequence of the middle fibration.  Therefore, we have 

\begin{lem}\label{lem:Tr} In $\{{_{LS}}E_r^{*,*}, d_r \}$, the element $z_i$ is transgressive to an element of the form $u_i\otimes 1 - 1\otimes u_i \in 
H^*(\M(\partial_{\text{in}}))$ for $i= 1, ..., l$, where $\beta^*u_i$ is identified with $u_i \in H^*(\M(\partial_{in}))$ via isomorphisms in (3.4), (3.5) and (3.6) for $i= 1 ,..., l$. 
\end{lem}

By \cite[Lemma 3.4]{S_3}, we can write 
$$(B\iota)^*x_i\otimes 1 - 1\otimes (B\iota)^*x_i  = \sum_{j=1}^l \zeta_{ij}(u_j\otimes 1- 1\otimes u_j)$$ 
in $H^*(BH)\otimes H^*(BH)$  for $i = 1, ..., l$ with elements $\zeta_{ij}$ in $H^*(BH)\otimes H^*(BH)$ which satisfy the condition that 
$\m(\zeta_{ij})= \frac{\partial (B\iota)^*x_i}{\partial u_j}$, where $\m$ denotes the multiplication on $H^*(BH)$. 
Then it follows from Lemma \ref{lem:Tr} and Zeeman's comparison theorem that an element of the form 
$w_i := \sum_j^l \zeta_{ij} z_j$ is 
a permanent cycle for $i = 1, ..., l$ and that $w_1, ..., w_l$ are generators in the vector space $(Q\text{Tot}_{LS}E_\infty^{*,*})^{\text{odd}}$ 
of the indecomposable elements of $\text{Tot}{}_{LS}E_\infty^{*,*}$ with odd degree. 
In fact, the model in Zeeman's comparison theorem above enables us to deduce that each $z_i$ is not in the image of the differential in a model of the spectral sequence. The computation in (3.4), (3.5) and (3.6) implies that 
\[
\K\{w_1, ..., w_l\} \cong (Q\text{Tot}{}_{LS}E_\infty^{*,*})^{\text{odd}} \cong (QH^*(\M(W)))^{\text{odd}} \cong \K\{y_1, ..., y_l\}. 
\]
As for the first isomorphism, it follows from the argument above that there exists a surjective map
$\K\{w_1, ..., w_l\} \to (Q\text{Tot}{}_{LS}E_\infty^{*,*})^{\text{odd}}$. The second and third isomorphisms yields 
that $\dim (Q\text{Tot}{}_{LS}E_\infty^{*,*})^{\text{odd}} =l$. We have the first isomorphism. 
It turns our that $0\neq y_1\cdots y_l = w_1\cdots w_l = \text{det}(\zeta_{ij})z_1\cdots z_l$ in ${}_{LS}E_\infty^{*,\dim H}$ changing the generators $y_1, ..., y_l$ 
if it is necessary.  This implies that 
$\mu_{W}(1\otimes y_1\cdots y_l) = h^!\circ k^*(1\otimes y_1\cdots y_l) = 
h^!(1\otimes y_1\cdots y_l) = h^!( \text{det}(\zeta_{ij})z_1\cdots z_l) = \text{det}(\zeta_{ik})$. 
The last equality follows from the definition of the integration. 
\end{proof}

\begin{proof}[Proof of the non-triviality of $\mu_{W^{\text{op}}}$] We consider the integration $k^!$ along the fibre associated with the fibration 
$$(*) : G/H \stackrel{i}{\to} 
\M(W^{\text{op}})=\M(W) \stackrel{k}{\to} \M(\partial_{\text{out}})=\M((\partial^{\text{op}})_{\text{in}})
$$
which is mentioned in Remark \ref{rem:fibre}. By virtue of \cite[3.5 Proposition]{B}, we see that $H^*(BH)\cong \K[(B\iota)^*x_1, ..., (B\iota)^*x_l]\otimes M$ as an  
$H^*(BG)$-module for some  graded vector space $M$. In particular, $H^*(\M(W))$ is a free 
$H^*( \M(\partial_{\text{out}}))$-module. Thus the argument of the EMSS for the fibration $(*)$ enables us to deduce that 
$i^* : M \stackrel{\cong}{\to} H^*(G/H)$ is an isomorphism; see \cite[Proposition 4.2]{S_2}. 

It follows from the definition of $k^!$ that $k^!(\Lambda_W \cdot \alpha \otimes y^{i_1}\cdots y^{i_l})= \alpha \otimes y^{i_1}\cdots y^{i_l}$, where 
$\alpha \in  \K[(B\iota)^*x_i]$, $i_k = 0$ or $1$ and $i^*(\Lambda_W)$ denotes the fundamental class of the homogeneous space $G/H$. Therefore, we see that 
$
D\mu_{W^{\text{op}}}(1\otimes \Lambda_W \cdot \alpha) = k^! \circ h^*(1\otimes \Lambda_W \cdot \alpha) = 
k^!(\Lambda_W \cdot \alpha) = \alpha 
$ 
for an element of the form 
$$1\otimes \Lambda_W \cdot \alpha \in 
\frac{H^*(BH)\otimes H^*(BH)}{((B\iota)^*x_i\otimes 1 - 1\otimes (B\iota)^*x_i \mid 1\leq i \leq l)}\cong H^*(\M(\partial_{\text{in}})).$$
Thus $D\mu_{W^{\text{op}}}$ is non-trivial in general. 
\end{proof}

\begin{proof}[Proof of the latter half of Theorem \ref{thm:main}]
It remains to show that the composite 
$$D\mu_{W^{\text{op}}}\circ D\mu_{W}=D(\mu_{W} \circ \mu_{W^{\text{op}}})=D(\mu_{W \circ W^{\text{op}}})$$ is non-trivial. 
By the assumption of the degree of $(B\iota)^*x_i$ for $i$ and the result \cite[Proposition 3]{S}, we see that 
$\text{det}(\frac{\partial (B\iota)^*x_i}{\partial u_j})$ is the fundamental class $\Lambda_W$ of $G/H$. Therefore, the computation above and the choice of elements 
$\zeta_{ij}$ allow us to conclude that 
\begin{eqnarray*}
D\mu_{W^{\text{op}}}\circ D\mu_{W}(y_1\cdots y_l) \!\!\! &=& \!\!\! D\mu_{W^{\text{op}}}(\text{det}(\zeta_{ij})) = 
k^!(\m(\text{det}(\zeta_{ij}))) \\
&=&  \!\!\! k^!(\text{det}(\m(\zeta_{ij}))= 
k^!\Big(\text{det}\Big(\frac{\partial (B\iota)^*x_i}{\partial u_j}\Big)\Big) =1. 
\end{eqnarray*}
This completes the proof. 
\end{proof}

We conclude this section with remarks on other operations obtained by a composite with the whistle cobordism. 
The results show a fruitful structure of the labeled TQFT for classifying spaces.  

\begin{rem}\label{rem:OtherOperations}
(i) The integration $h^!$ along the fibre is a morphism of $H^*(\M(\partial_{\text{in}}))$-modules via $h^*$. Thus the computation of $\mu_W$ above yields that 
for $\gamma \in H^*(BG)$, 
\begin{eqnarray*}
\mu_W(\gamma \otimes y^{i_1}\cdots y^{i_l}) &=& \!\! h^!((B\iota)^*\gamma \otimes y^{i_1}\cdots y^{i_l}) = (1\otimes (B\iota)^*\gamma) h^!(1\otimes y^{i_1}\cdots y^{i_l}) \\
&=& \!\!
\begin{cases}
(1\otimes (B\iota)^*\gamma) \text{det}(\zeta_{ij}) & \text{if} \ i_1\cdots i_l = 1 \\
0 & \text{otherwise.}
\end{cases}
\end{eqnarray*}

(ii) Since the image of $D\mu_{W^{\text{op}}}$ is in $H^*(BG)\otimes 1$, it follows from (i) that 
$$D(\mu_{W^{\text{op}} \circ W})= D(\mu_{W^{\text{op}}}\circ \mu_{W}) = 
D\mu_{W}\circ D\mu_{W^{\text{op}}}= 0. $$
This yields that in the labeled TQFT for the classifying space, 
the operation for a cobordism with $n$ holes labeled by connected closed subgroups of maximal rank is trivial provided 
$n \geq 2$ and the characteristic of the underlying field is sufficiently large. As a consequence, under the same assumption, the Cardy condition \cite[(2.14)]{L-P} implies 
that the cobordism operation associated with the double-twist diagram \cite[Fig .11]{M-S} in the open TQFT is trivial. This result also follows form Theorem \ref{thm:openstrings} below in which the open theory is clarified in our setting.  

(iii)  Let $\Sigma$ be the pair of pants with one incoming boundary.  
The result \cite[Theorem 4.1]{K-M}, in which each $y_j$ is replaced with the notation $x_j$, 
enables us to deduce that the operator $\mu_{\Sigma \circ W}=\mu_{\Sigma} \circ \mu_W$ is non-trivial in general. 
 
(iv) We can consider the whistle cobordism operation in a homological conformal field theory (HCFT). 
Let $\bigoplus H_*(BDiff^+(\Sigma; \partial))$ be the prop parameterized by the homology of mapping class groups. By using the prop, we have a HCFT structure for classifying spaces; see \cite{C-M, G}. Let 
$C$ be the cylinder $S^1 \times [0, 1]$ and 
$$\circ : H_*(BDiff^+(C; \partial))\otimes H_*(BDiff^+(W; \partial))\to 
H_*(BDiff^+(C\circ W; \partial))$$ 
the prop structure coming from the gluing of bordisms. The Dehn twist gives rise to the 
element $\Delta$ in $H_1(BDiff^+(C\circ W; \partial))$ via the Hurewicz map. In fact, the element $\Delta$ 
induces the Batalin-Vilkovisky (B-V) operator on $H^*(LBG)$ 
by the HCFT structure; see \cite[Proposition 60]{C-M}. Observe that $\mu_W$ is regarded as an element in 
$H_0(BDiff^+(W; \partial))$.  Then under the same assumption as in (ii) 
and with the notation in the proof of Theorem \ref{thm:main}, we see that 
\[
D(\Delta \circ \mu_W)(x_1y_2\cdots y_l)= (D\mu_W \circ D\Delta) (x_1y_2\cdots y_l) = 
D\mu_W(y_1\cdots y_l) \neq 0. 
\]
Observe that the B-V operator is a derivation 
with respect to the cup product on the cohomology. Then the second equality follows from \cite[Theorem 3.1]{K-M}. 
\end{rem}

Under the same assumption as in Remark \ref{rem:OtherOperations} (ii), the cobordism operation associated to 
the pair of pants with two incoming boundaries is trivial on $H_*(LBG)$; see \cite[Theorems 7.1 and 7.3]{K-M}. 
Thus we can exactly understand the closed TQFT structure for classifying spaces. 
Thanks to the main theorem, we can also compute every cobordism operation in the open-closed TQFT labeled in $\mathcal{B}$ the set of connected closed subgroups 
of maximal rank if the open TQFT is clarified. The consideration of the open theory is the topic in Appendix A. 

\medskip
\noindent
{\it Acknowledgments.} The author is grateful to Anssi Lahtinen for inviting him to University of Copenhagen, 
and also thanks Jesper Grodal. The inspired discussions with them on string topology have enabled 
the author to refine the computations in this manuscript. 

\section{Appendix A: A labelled open TQFT for classifying spaces} 

Let $G$ be a connected compact Lie group whose cohomology with coefficients in $\K$ is a polynomial algebra over generators with even degree.   
Let $K$, $H$ and $L$ be connected closed subgroup of $G$ of maximal rank, 
whose cohomology algebras satisfy the same condition 
as that of $G$. Then we have  

\begin{thm}\label{thm:openstrings}
Let $\Upsilon$ be the basic cobordism from two labeled intervals $I^K_H$ and $I^H_L$ to one labeled interval $I^K_L$, which is pictured  
in \cite[Fig. 1]{M-S}. Then the cobordism operation $\mu_{\Upsilon}$ is trivial but not $\mu_{\Upsilon^{\text{\em op}}}$ in general.  
\end{thm}

Appendix A is devoted to proving the theorem. We consider fibrations below which define the cobordism operations $\mu_\Upsilon$ and $\mu_{\Upsilon^{\text{op}}}$. 
Moreover, we investigate the cohomology algebras of total and base spaces. The results are described with the diagram  
\[
\xymatrix@C12pt@R20pt{
\M(I^K_L)  & H^*(\M(I^K_L)) \ar[d]_-{((out)^*)^*} &
     \frac{H^*(BK)\otimes H^*(BL)}{((Bk)^*x_i\otimes 1 - 1\otimes (B\ell)^*x_i )} \ar[d]_\varphi \ar[l]_-\cong \\
    \M(\Upsilon) \ar[u]_-{(out)^*}  \ar[d]^-{(in)^*}  & H^*(\M(\Upsilon)) &
    \frac{H^*(BK)\otimes  H^*(BH) \otimes H^*(BL)}{((Bk)^*x_i\otimes 1 - 1\otimes (B\iota)^*x_i,  
      (B\iota )^*x_i\otimes 1 - 1\otimes (B\ell )^*x_i)}  \ar[l]_-\cong\\
\! \M(I^K_H \coprod I^H_L) & \! \! \! \! \!  H^*( \M(I^K_H \coprod I^H_L))  \ar[u]^-{((in)^*)^*}& 
\frac{H^*(BK)\otimes  H^*(BH)}{((Bk)^*x_i\otimes 1 - 1\otimes (B\iota)^*x_i)}
   \!  \otimes \!
      \frac{H^*(BH)\otimes  H^*(BL)}{((B\iota)^*x_i\otimes 1 - 1\otimes (B\ell)^*x_i)}  \ar[u]^-{1\otimes \m \otimes 1}, 
      \ar[l]_-\cong
 }
 \eqnlabel{add-6}
\]
where $\varphi$ is the map induced by the natural map 
\[
H^*(BK)\otimes H^*(BL) \to H^*(BK)\otimes 1 \otimes H^*(BL) \to H^*(BK)\otimes H^*(BH)\otimes H^*(BL), 
\] 
and $k : K \to G$, $\iota : H \to G$ and $\ell : L \to G$ denote the inclusions. 
Observe that the fibres of the fibrations $(out)^*$ and  $(in)^*$ are $G/H$ and $\Omega BH \simeq H$, 
respectively; see the proof of \cite[Lemma 2.3.11]{G}. The cohomology algebras are computed with 
the Eilenberg-Moore spectral sequence as made in Section 3 associated with pullback diagrams defining the spaces  $\M(I^K_L)$,  $\M(\Upsilon)$ and $\M(I^K_H \coprod I^H_L)$. 

The computation of $H^*(\M(\Upsilon))$ is here given.  
We choose a homotopy equivalence 
\[
\objectmargin={0pt}
\xymatrix@C15pt@R15pt{
r : \partial_{in} = a_1 \ar@{-}[r] & \bullet \ar@{-}@/^0.7pc/[r] & a_2 \ar@{-}[r] & a_3  \ \ \ \ar[r]^\simeq & \ \ \ 
 a_1 \ar@{-}[r] & a_2 \ar@{-}[r] & a_3 = I \cup_{a_2} I
}
\]
with a homotopy inverse $i$ which satisfies the condition that $i(a_s) = a_s$ for $s = 1,2$ and $3$. 
We regard $\Upsilon$ as a labeled cobordism $(\Upsilon, \{\Upsilon^{H_i}\}_{i=1,2,3}\})$, where $H_1= K$, $H_2= H$, $H_3= L$. 
Observe that the in-boundaries of the free boundaries $\Upsilon^{H_1}$, $\Upsilon^{H_2}$ and $\Upsilon^{H_3}$ are 
the sets $\{a_1\}$, $\{\bullet, a_2\}$ and $\{a_3\}$, respectively. By the definition of the space $\M(\Upsilon)$, we have commutative diagrams  
\[
\xymatrix@C15pt@R20pt{
\M(\Upsilon) \ar[r] \ar[d] & \text{map}(\Upsilon, BG) \ar[r]^{i^*}_{\simeq} \ar[d] & BG^{I \cup_{a_2} I} \ar[d]^{(ev_0, ev_{\frac{1}{2}}, ev_{1})} \ar[r] & BG^I \times BG^I \ar[d]^{(ev_0\times ev_1)^{\times 2}}\\
\prod_{i=1}^3 \text{map}(\Upsilon^{H_i}, BH_i)  \ar[r]_\psi &  \prod_{i=1}^3 \text{map}(\Upsilon^{H_i}, BG) 
\ar[r]_-{\gamma}^-{\simeq} & BG^{\times 3} \ar[r]_{1\times \Delta \times 1} & BG^{\times 4}
}
\]
in which the right-hand side and left-hand side squares are pullback diagrams, 
where the map $\gamma$ is defined by the embeddings which are homotopy inverses of 
deformation retractions $\Upsilon^{H_i} \to \{a_i\}$ and 
$\psi$ is the map induced by the inclusions $k$, $\iota$ and $\ell$ mentioned above. Thus applying the EMSS to the big square 
which is a homotopy pullback, we can compute the cohomology algebra $H^*(\M(\Upsilon))$ in (4.1) with the two sided  Koszul resolution mentioned in Section 3. 

\begin{proof}[Proof of Theorem \ref{thm:openstrings}]
By definition, we see that $D\mu_{\Upsilon} = ((in)^*)^!\circ ((out)^*)^*$. The image of $((out)^*)^*$ is in the image of 
$((in)^*)^*$. This follows from the commutativity of the diagram (4.1). Then the definition of the integration 
along the fibre yields that $D\mu_{\Upsilon}$ is trivial. 

We investigate the Leray-Serre spectral sequence for the fibration $G/H \stackrel{j}{\longrightarrow}  \M(\Upsilon) \stackrel{(out)^*}{\longrightarrow} \M(I^K_L)$.  
Since the subgroup $H$ is of maximal rank, it follows that 
the $E_2$ term is generated by elements with even degree and hence  the map $j^*$ induced by the inclusion 
$j$ is an epimorphism. Therefore, there exists an element $\Lambda_{\Upsilon}$ 
of the form $\sum_i q_i \cdot (1\otimes b_i \otimes 1)$ in 
$H^*(\M(\Upsilon))$ with $q_i \in H^*(BK)\otimes  H^*(BL)$ and  $b_i \in H^*(BH)$ 
such that $j^*(\Lambda_{\Upsilon})$ is the fundamental class of $G/H$. We can assume that $\deg q_i = 0$ for any $i$ 
because $q_i$ is in the image of $\varphi$. Thus we see that 
$D\mu_{\Upsilon^{\text{op}}} = ((out)^*)^!\circ ((in)^*)^*(\sum_i q_i \cdot (1\otimes b_i \otimes 1 \otimes 1))= 
((out)^*)^!(\sum_i q_i \cdot (1\otimes b_i \otimes 1))= 1$. This completes the proof. 
\end{proof}

By virtue of Theorems \ref{thm:main}, \ref{thm:openstrings}, the results \cite[Theorems 4.1 and 7.1]{K-M} 
for closed TQFT for classifying spaces and 
the result \cite[Proposition 3.9]{L-P}, we have 

\begin{assertion}\label{assertion:TQFT} 
Let $\B$ be the set of connected closed subgroup of $G$ of maximal rank. Then one can make a calculation of each of the dual operations for the labeled TQFT $\mu :  (\mathsf{oc\text{-}Cobor}(\B), \coprod) \to 
(\mathbb{Q}\text{-}\mathsf{Vect}, \otimes)$ introduced by Guldberg 
up to multiplication by non-zero scalar with the cohomology algebras and their generators described 
in (3.6) and (4.1), and moreover, with representatives $\Lambda_W$ and $\Lambda_{\Upsilon}$ of the fundamental classes of the homogeneous spaces $G/H$ in $H^*(\M(W))$ and $H^*(\M(\Upsilon))$; see the proof of Theorems \ref{thm:main} and 
\ref{thm:openstrings} for $\Lambda_W$ and $\Lambda_{\Upsilon}$.
\end{assertion}

In particular, we see that $\mu_\Sigma \equiv 0$ for a cobordism $\Sigma$ which has two holes, 
or contains at least either one of $\Upsilon$ and the pair of pants with two in-boundaries as a component constructing the cobordism with gluing.

\section{Appendix B : An algebraic model for the whistle cobordism operation}

We give an algebraic model for $\mu_W$ in cochain level over the rational. In this section, the cochain algebra $C^*(X)$ for a space $X$ is regarded as the DG algebra of PL differential forms on $X$ thought the same notation as that of the singular cochain algebra is used; see \cite{B-G} and \cite[Section 10]{FHT1} for PL differential forms. In what follows, we assume that $H$ is an arbitrary connected closed subgroup of a connected compact Lie group $G$. 

We recall the commutative diagram in (3.2) and first consider the fibration $res :  \text{map}({}_a\cap_b, BH) \to \text{map}(\{ a, b\}, BH)$ in the left square. 
The results \cite[Theorems 14.12 and 15.3]{FHT1} allow us to obtain a minimal relative Sullivan model for $res$ 
of the form 
$$\zeta : C^*(\text{map}(\{ a, b\}, BH)) \otimes H^*(\Omega BH) \stackrel{\simeq}{\longrightarrow} C^*(\text{map}({}_a\cap_b, BH)).$$
Observe that the source is the tensor product as a vector space, but not as a DGA. 
By using the model $\zeta$, we have a model $res^! : C^*(\text{map}(\{ a, b\}, BH)) \otimes H^*(\Omega BH) \to 
C^*(\text{map}(\{ a, b\}, BH))$  
for the integration along the fibre of the fibration $res$ mentioned above; see \cite[Theorem 5]{F-T}. The left-hand side diagram in (3.2) is the pullback described in Remark \ref{rem:fibre}. Then the proof of \cite[Theorem 6]{F-T} yields that  $res^!\otimes 1$ in (5.1) below is a model for the integration $h^!$. 
Moreover, a quasi-isomorphism 
$$u : C^*(\text{map}(\{ a, b\}, BH))\otimes^{{\mathbb L}}_{C^*(\text{map}({}_a\cap_b, BG))}C^*(\text{map}({}_a\text{--}_b, BG) ) \stackrel{\simeq}{\to} C^*(\M(\partial_{\text{in}}))$$
is induced by the front pullback in (3.2). 
Thus we have commutative diagrams with solid arrows 
$$
{\small 
\xymatrix@C18pt@R15pt{
C^*(\M(\partial_{\text{out}})) \ar[dd]_{k^*} & \Q \otimes_\Q C^*(\text{map}(S^1, BG)) \ar[d]^{\eta\otimes res^*}  
       \ar[l]_{\simeq} \\
                  & C^*(\text{map}({}_a\cap_b, BH)) \otimes^{\mathbb{L}}_{C^*(\text{map}({}_a\cap_b, BG))} 
                   C^*(\text{map}(W, BG) ) \ar[dl]_-{\simeq}^-{\xi_1} \ar@{..>}@/^2pc/[dd]_-{\Psi} \\
C^*(\M(W))  \ar[dd]_{h^!}& \\
                  &  C^*(\text{map}({}_a\cap_b, BH)) \otimes^{\mathbb{L}}_{C^*(\text{map}({}_a\cap_b, BG))} 
                   C^*(\text{map}({}_a\text{--}_b, BG) ) \ar[ul]^-{\simeq}_-{\xi_2} \ar[uu]^-{1\otimes_{res^*}res^*} \ar[d]^{1\otimes u}_{\simeq} \\ 
C^*(\M(\partial_{\text{in}})) & C^*(\text{map}({}_a\cap_b, BH)) \otimes^{\mathbb{L}}_{C^*(\text{map}(\{ a, b\}, BH))} 
                  C^*(\M(\partial_{\text{in}})) \ar@/^2pc/[uul]^-{\simeq}_(0.8){\xi_3} \ar[l]^-{res^!\otimes 1}
}
}
\eqnlabel{add-3}
$$
in which $\xi_1$ and $\xi_3$ are quasi-isomorphisms induced by the back pullback diagram and the left-hand side pulback diagram in (3.2), respectively, and 
$\xi_2$ is a quasi-isomorphism induced by the big pullback which the left-hand side and the front pullback diagrams give. 
By the lifting lemma \cite[Proposition 12.9]{FHT1} enables us to obtain a right inverse $\Psi$ of $1\otimes_{res^*}res^*$ in the derived category of 
$C^*(\text{map}({}_a\cap_b, BG))$-modules. 
The last step in constructing the model for $\mu_W=h^!\circ k^*$ is not explicit. In fact, as expected from 
the proof of Theorem \ref{thm:main}, 
it seems that the construction of the lift is complicated in general. 
Under appropriate assumptions on $G$ and the subgroup $H$, it is anticipated that 
Sullivan models serve the explicit calculation. However, we do not pursue the topic in this article.

\end{document}